\documentclass[12pt,a4paper,twoside,english,american]{scrartcl}
\usepackage{helvet}
\usepackage[T1]{fontenc}
\usepackage[latin9]{inputenc}
\usepackage{babel}
\usepackage{mathtools}
\usepackage{amsthm}
\usepackage{amsmath}
\usepackage{amssymb}
\usepackage{esint}
\usepackage[unicode=true,pdfusetitle,
 bookmarks=true,bookmarksnumbered=false,bookmarksopen=false,
 breaklinks=false,pdfborder={0 0 1},backref=false,colorlinks=false]
 {hyperref}

\makeatletter

\pdfpageheight\paperheight
\pdfpagewidth\paperwidth

\theoremstyle{plain}
\newtheorem{thm}{\protect\theoremname}[section]
  \theoremstyle{plain}
  \newtheorem{cor}[thm]{\protect\corollaryname}
  \theoremstyle{plain}
  \newtheorem{lem}[thm]{\protect\lemmaname}
  \theoremstyle{remark}
  \newtheorem{rem}[thm]{\protect\remarkname}

\@ifundefined{date}{}{\date{}}

\usepackage{a4wide}

\usepackage{euscript}
\allowdisplaybreaks[1] 
\usepackage{amsfonts}\newenvironment{keywords}{ \noindent\footnotesize\textbf{Keywords and phrases:}}{}

\newenvironment{class}{\noindent\footnotesize\textbf{Mathematics subject classification 2010:}}{}

\newcommand*{\dive}{\operatorname{div}}

\newcommand*{\grad}{\operatorname{grad}}

\newcommand*{\ii}{\mathrm{i}}
\renewcommand*{\i}{\mathrm{i}}

\DeclareMathAccent{\Circ}{\mathalpha}{operators}{"17}

\renewcommand{\Re}{\operatorname{\mathfrak{Re}}}

\renewcommand{\hat}{\widehat}

\renewcommand*{\epsilon}{\varepsilon}
\renewcommand*{\theta}{\vartheta}
\renewcommand*{\nu}{\varrho}
\renewcommand*{\rho}{\varrho}

\usepackage{babel}

\makeatother

  \addto\captionsamerican{\renewcommand{\corollaryname}{Corollary}}
  \addto\captionsamerican{\renewcommand{\lemmaname}{Lemma}}
  \addto\captionsamerican{\renewcommand{\remarkname}{Remark}}
  \addto\captionsamerican{\renewcommand{\theoremname}{Theorem}}
  \addto\captionsenglish{\renewcommand{\corollaryname}{Corollary}}
  \addto\captionsenglish{\renewcommand{\lemmaname}{Lemma}}
  \addto\captionsenglish{\renewcommand{\remarkname}{Remark}}
  \addto\captionsenglish{\renewcommand{\theoremname}{Theorem}}
  \providecommand{\corollaryname}{Corollary}
  \providecommand{\lemmaname}{Lemma}
  \providecommand{\remarkname}{Remark}
\providecommand{\theoremname}{Theorem}

\begin{document}
\selectlanguage{english}%

\selectlanguage{american}%




\setcounter{section}{-1}

\title{On Maximal Regularity for a Class of Evolutionary Equations.}

\author{Rainer Picard, Sascha Trostorff\\
 {\footnotesize{}Institut für Analysis, Fachrichtung Mathematik, Technische
Universität Dresden}\\
 Marcus Waurick\\
 {\footnotesize{}Department of Mathematical Sciences, University of
Bath}\\
 {\tiny{}rainer.picard@tu-dresden.de, sascha.trostorff@tu-dresden.de,
marcus.waurick@tu-dresden.de}}
\maketitle
\begin{abstract}
The issue of so-called maximal regularity is discussed within a Hilbert
space framework for a class of evolutionary equations. Viewing evolutionary
equations as a sums of two unbounded operators, showing maximal regularity
amounts to establishing that the operator sum considered with its
natural domain is already closed. For this we use structural constraints
of the coefficients rather than semi-group strategies or sesqui-linear
form methods, which would be difficult to come by for our general
problem class. Our approach, although limited to the Hilbert space
case, complements known strategies for approaching maximal regularity
and extends them in a different direction. The abstract findings are
illustrated by re-considering some known maximal regularity results
within the framework presented. 
\end{abstract}

\begin{class} 35B65, 35D35, 58D25 \end{class}

\tableofcontents{}

\section{Introduction}

The issue of maximal regularity has received much attention as an
important property of certain partial differential equations and more
abstractly as a feature of a class of evolution equations. In a Hilbert
space setting, the typical situation thus refers to an abstract operator
equation in $L^{2,\mathrm{loc}}([0,\infty[,H)$ of the form 
\begin{equation}
u'+\mathcal{A}u=f,\label{eq:mr1}
\end{equation}
for some given $f\in L^{2,\mathrm{loc}}([0,\infty[,H)$, $H$ a Hilbert
space. Moreover, $u\colon]0,\infty[\to H$ is a measurable function
with $u'$ being its weak derivative and $\mathcal{A}$ is the (abstract)
linear operator on $L^{2,\mathrm{loc}}([0,\infty[,H)$ induced by
an operator $A$, assumed to be the infinitesimal generator of a one-parameter
$C_{0}$-semi-group, i.e. $\left(\mathcal{A}u\right)(t)\coloneqq A(u(t))$.
If we solve equation \eqref{eq:mr1} for $u$ subject to homogeneous
initial conditions, we can expect $u$ to be at best only continuous.
Thus, $u$ is a so-called \emph{mild solution of \eqref{eq:mr1}},
that is, $u$ solves the equation in question in an integrated form.
To obtain better regularity behaviour one is interested in the case,
where for any given $f$, the corresponding solution $u$ is such
that $u'$ and $\mathcal{A}u$ both belong to $L^{2,\mathrm{loc}}([0,\infty[,H)$
and, hence, $u$ \emph{literally} solves \eqref{eq:mr1} in $L^{2,\mathrm{loc}}([0,\infty[,H)$.
This property is commonly attributed to the semi-group generator $A$
and one says in this case that $A$ \emph{admits maximal $L^{2}$-regularity}.
A standard situation is that $\mathcal{A}$ is a non-negative selfadjoint
operator and so, if $\mathcal{A}=C^{*}C$ for some closed and densely
defined linear operator $C$, the corresponding evolution equation
admits maximal regularity as can be easily seen in this simple case
with the help of the spectral theorem for $A$. We shall refer to
the seminal paper \cite{daPrato1975} as a standard reference for
maximal regularity. We also refer the reader to \cite{Chill2005,Chill2008,Batty2015}
for the $L^{p}$-maximal regularity of second-order Cauchy problems,
to \cite{Arendt2007,Arendt2014} for maximal regularity for non-autonomous
problems, to \cite{Sforza1995,Zacher2005} for integro-differential
equations and to \cite{Bu2011,Ponce2013} for fractional differential
equations.

In this article, we revisit the standard Hilbert space case $\mathcal{A}=C^{*}C$
under a system perspective: By setting $v\coloneqq-Cu$ we deduce
from \eqref{eq:mr1}, writing $\partial_{0}$ for the time derivative,
the operator equation 
\[
\left(\begin{array}{cc}
\partial_{0} & 0\\
0 & 0
\end{array}\right)\left(\begin{array}{c}
u\\
v
\end{array}\right)+\left(\begin{array}{cc}
0 & 0\\
0 & 1
\end{array}\right)\left(\begin{array}{c}
u\\
v
\end{array}\right)+\left(\begin{array}{cc}
0 & -C^{*}\\
C & 0
\end{array}\right)\left(\begin{array}{c}
u\\
v
\end{array}\right)=\left(\begin{array}{c}
f\\
0
\end{array}\right).
\]
Now, we ask for the maximal regularity, when the coefficients $\left(\begin{array}{cc}
\partial_{0} & 0\\
0 & 0
\end{array}\right)$ and $\left(\begin{array}{cc}
0 & 0\\
0 & 1
\end{array}\right)$ are replaced by more general operators $\left(\begin{array}{cc}
\partial_{0}\mathcal{M} & 0\\
0 & 0
\end{array}\right)$ and $\mathcal{N}$ acting in space time. Under suitable conditions
on $\mathcal{M}$ and $\mathcal{N}$, we will show in our main Theorem
\ref{thm:maxreck}, that for a given $L^{2}$-right-hand side $f$,
the solution $(u,v)$ has the following properties. We have that $u$
is weakly $L^{2}$-differentiable with respect to time and that $(u,v)\in D\left(\left(\begin{array}{cc}
0 & -C^{*}\\
C & 0
\end{array}\right)\right)$. Moreover, the equation 
\[
\left(\begin{array}{cc}
\partial_{0}\mathcal{M} & 0\\
0 & 0
\end{array}\right)\left(\begin{array}{c}
u\\
v
\end{array}\right)+\mathcal{N}\left(\begin{array}{c}
u\\
v
\end{array}\right)+\left(\begin{array}{cc}
0 & -C^{*}\\
C & 0
\end{array}\right)\left(\begin{array}{c}
u\\
v
\end{array}\right)=\left(\begin{array}{c}
f\\
0
\end{array}\right)
\]
is satisfied \emph{literally.} This remains true if we alter the right-hand
side $\left(\begin{array}{c}
f\\
0
\end{array}\right)$ to $\left(\begin{array}{c}
f\\
g
\end{array}\right)$ for any weakly differentiable $g$. Our first order approach complements
known results on maximal regularity by allowing for quite general
coefficients $\mathcal{M},\mathcal{N}$. With this generalization,
we enter the realm of so-called evolutionary equations, which we briefly
introduce in the next section. This class comprises the standard initial
boundary value problems of mathematical physics in a unified setting,
we refer to \cite{Picard2014_survey} for a survey. After having introduced
the mathematical framework, we will provide our main result in Section
\ref{sec:The-main-result}. We conclude this article with several
illustrative examples in the last section. The more involved examples
are (abstract) second order problems (in both time and space) (adapted
from \cite{Batty2015,Sforza1995}) as well as problems with a fractional
time derivative, which is an adaptation from \cite{Ponce2013}.

\section{\label{sec:A-brief-description}A brief description of the framework
of evolutionary equations}

We recall the notion of evolutionary equations, as introduced in \cite[Solution Theory]{Picard},
a term we use in distinction to classical evolution equations, which
are a special case. For this, let throughout $\nu$ be a positive,
real parameter and $H$ a Hilbert space. Define 
\[
L_{\nu}^{2}(\mathbb{R},H)\coloneqq\{f\in L_{\textnormal{loc}}^{2}(\mathbb{R},H)|(t\mapsto e^{-\nu t}f(t))\in L^{2}(\mathbb{R},H)\},
\]
which endowed with the natural scalar product 
\[
\langle f,g\rangle\coloneqq\intop_{\mathbb{R}}\langle f(t),g(t)\rangle_{H}e^{-2\rho t}\mbox{ d}t\quad(f,g\in L_{\rho}^{2}(\mathbb{R},H))
\]
is again a Hilbert space. The operator 
\[
\partial_{0}\colon D(\partial_{0})\subseteq L_{\nu}^{2}(\mathbb{R},H)\to L_{\nu}^{2}(\mathbb{R},H),f\mapsto f'
\]
with $f'$ being the distributional derivative and $D(\partial_{0})=\{f\in L_{\nu}^{2}(\mathbb{R},H)|f'\in L_{\nu}^{2}(\mathbb{R},H)\}$
defines a normal operator with $\Re\partial_{0}=\nu$ (see e.g \cite[Section 2.2]{Picard2014_survey}).
Indeed, $\partial_{0}$ is unitarily equivalent to the operator $\ii m+\nu$
of multiplication by the function $\xi\mapsto\ii\xi+\nu$ considered
as an operator in $L^{2}(\mathbb{R},H)$. This spectral representation
result is realized by the so-called Fourier--Laplace transformation
$\mathcal{L}_{\nu}\colon L_{\nu}^{2}(\mathbb{R},H)\to L^{2}(\mathbb{R},H)$,
that is, the unitary extension of the integral operator given by 
\[
\mathcal{L}_{\nu}\phi(\xi)\coloneqq\frac{1}{\sqrt{2\pi}}\int_{\mathbb{R}}e^{-\ii t\xi-\nu t}\phi(t)\mbox{ d}t\quad(\xi\in\mathbb{R})
\]
for bounded, measurable and compactly supported functions $\phi\colon\mathbb{R}\to H.$
In particular, since $\rho>0,$ we read off that $\partial_{0}$ is
boundedly invertible on $L_{\rho}^{2}(\mathbb{R},H)$ with $\|\partial_{0}^{-1}\|\leq\frac{1}{\rho}.$
It is clear that the spectrum of $\ii m+\nu$ is given by the set
$\ii\left[\mathbb{R}\right]+\nu$. Hence, $\sigma(\partial_{0}^{-1})=\sigma(\left(\ii m+\nu\right)^{-1})=\partial B_{\mathbb{C}}(r,r)$
with $r=1/(2\nu)$. Thus, the said spectral representation gives rise
to a functional calculus for $\partial_{0}^{-1}$: Let $r'>r$. For
an analytic bounded function $M\colon B_{\mathbb{C}}(r',r')\to L(H)$
we define 
\[
M(\partial_{0}^{-1})\coloneqq\mathcal{L}_{\nu}^{*}M\left(\left(\ii m+\nu\right)^{-1}\right)\mathcal{L}_{\nu},
\]
where $\left(M\left(\left(\ii m+\nu\right)^{-1}\right)\phi\right)(t)\coloneqq M\left(\left(\ii t+\nu\right)^{-1}\right)\phi(t)$
for all $t\in\mathbb{R}$ and $\phi\in L^{2}(\mathbb{R},H)$. Again,
we refer to \cite{Picard2014_survey} for several examples of analytic
operator-valued functions of $\partial_{0}^{-1}$ and their occurrence
in the context of partial differential equations.

The solution theory, that is, unique existence of solutions and continuous
dependence on the data, for many linear equations of mathematical
physics is covered by the following theorem. For this, note that we
do not distinguish between operators defined on $H$ and there respective
lifts to the space $L_{\nu}^{2}(\mathbb{R},H)$. Also the explicit
dependence on $\rho$ is frequently suppressed. 
\begin{thm}[{\cite[Solution Theory]{Picard},\cite[Theorem 6.2.5]{Picard_McGhee}}]
\label{thm:st} Let $A\colon D(A)\subseteq H\to H$ be skew-selfadjoint,
$M$ as above. Assume there is $c>0$ such that 
\[
\Re\langle z^{-1}M(z)\phi,\phi\rangle\geq c\langle\phi,\phi\rangle\quad(\phi\in H,z\in B_{\mathbb{C}}(r',r')).
\]
Then the operator $B\coloneqq\partial_{0}M(\partial_{0}^{-1})+A$
defined on its natural domain is closable and the closure is continuously
invertible, that is, $S\coloneqq\overline{B}^{-1}\in L(L_{\nu}^{2}(\mathbb{R},H))$.
Moreover, $S$ commutes with $\partial_{0}^{-1}$ and for all $u\in D(\overline{B}),$
and $\epsilon>0$, we have $\left(1+\epsilon\partial_{0}\right)^{-1}u\in D(B)=D(\partial_{0}M(\partial_{0}^{-1}))\cap D(A)$. 
\end{thm}
For the last statement of the theorem one may also consult \cite[Lemma 5.2]{Waurick2015_nonauto}.
We have purposely left out the reference to causality, which also
holds and is essential for well-posedness of evolutionary equations
in general, but plays a lesser role in this paper. We note the following
corollary to Theorem \ref{thm:st}. 
\begin{cor}
\label{cor:mrst}With the assumptions and notations in the last theorem,
the following is true. Let $u\in D(\overline{B}).$ If $u\in D(\partial_{0}M(\partial_{0}^{-1}))$,
then $u\in D(A)$ and $\overline{B}u=Bu=\partial_{0}M(\partial_{0}^{-1})u+Au.$\end{cor}
\begin{proof}
Let $\epsilon>0$ and define $u_{\epsilon}\coloneqq(1+\epsilon\partial_{0})^{-1}u$.
By Theorem \ref{thm:st}, we get $u_{\epsilon}\in D(B)$ and, since
$S$ commutes with $\partial_{0}^{-1}$, $(1+\epsilon\partial_{0})^{-1}\overline{B}u=Bu_{\epsilon}$.
Thus, since $(1+\epsilon\partial_{0})^{-1}\to1$ as $\epsilon\to0$
in the strong operator topology, we infer $u_{\epsilon}\to u$ and
$Bu_{\epsilon}\to\overline{B}u$ in $L_{\nu}^{2}(\mathbb{R},H)$ as
$\epsilon\to0$. Furthermore, from 
\[
Bu_{\epsilon}=\left(\partial_{0}M(\partial_{0}^{-1})+A\right)u_{\epsilon}=\partial_{0}M(\partial_{0}^{-1})u_{\epsilon}+Au_{\epsilon}=(1+\epsilon\partial_{0})^{-1}\partial_{0}M(\partial_{0}^{-1})u+Au_{\epsilon},
\]
we read off by the closedness of $A$, that $u\in D(A)$ and $\overline{B}u=\partial_{0}M(\partial_{0}^{-1})u+Au$. 
\end{proof}

\section{The main result\label{sec:The-main-result}}

In this section, we show a maximal regularity result for a prototype
equation (see also Corollary \ref{cor:MR2nd} below). Let throughout
this section $C\colon D(C)\subseteq H_{0}\to H_{1}$ be a densely
defined, closed linear operator between Hilbert spaces $H_{0}$ and
$H_{1}$, $r>0$. Moreover, let $M\colon B_{\mathbb{C}}(r,r)\to L(H_{0}),$
$N_{ij}\colon B_{\mathbb{C}}(r,r)\to L(H_{j},H_{i})$ analytic and
bounded, $i,j\in\{0,1\}$. The prototype operator to study in the
following is 
\begin{equation}
B\coloneqq\left(\partial_{0}\left(\begin{array}{cc}
M(\partial_{0}^{-1}) & 0\\
0 & 0
\end{array}\right)+\left(\begin{array}{cc}
N_{00}(\partial_{0}^{-1}) & N_{01}(\partial_{0}^{-1})\\
N_{10}(\partial_{0}^{-1}) & N_{11}(\partial_{0}^{-1})
\end{array}\right)\right)+\left(\begin{array}{cc}
0 & -C^{*}\\
C & 0
\end{array}\right)\label{eq:prob}
\end{equation}
with domain $D(\partial_{0}M(\partial_{0}^{-1}))\cap D\left(\left(\begin{array}{cc}
0 & -C^{*}\\
C & 0
\end{array}\right)\right)$ in the space $L_{\nu}^{2}(\mathbb{R},H_{0}\oplus H_{1})$, where
$\nu>1/2r$.

We will use the following assumptions 
\begin{enumerate}
\item \label{enu:ass_m_n}There is $c_{0}>0$ such that for all $z\in B_{\mathbb{C}}(r,r)$
and $\left(\phi,\psi\right)\in H_{0}\oplus H_{1}$ the estimate 
\[
\Re\left\langle \left(z^{-1}\left(\begin{array}{cc}
M(z) & 0\\
0 & 0
\end{array}\right)+\left(\begin{array}{cc}
N_{00}(z) & N_{01}(z)\\
N_{10}(z) & N_{11}(z)
\end{array}\right)\right)\left(\begin{array}{c}
\phi\\
\psi
\end{array}\right),\left(\begin{array}{c}
\phi\\
\psi
\end{array}\right)\right\rangle \geq c_{0}\left\langle \left(\begin{array}{c}
\phi\\
\psi
\end{array}\right),\left(\begin{array}{c}
\phi\\
\psi
\end{array}\right)\right\rangle 
\]
is satisfied. 
\item For some $\beta\in]0,1]$ we have

\begin{enumerate}
\item \label{enu:ass_m}There is $c_{1}>0$ such that for all $z\in B_{\mathbb{C}}(r,r)$
and $\phi\in H_{0}$ the estimate 
\[
\Re\langle z^{\beta-1}M(z)\phi,\phi\rangle\geq c_{1}\langle\phi,\phi\rangle
\]
is satisfied and the mapping $B_{\mathbb{C}}(r,r)\ni z\mapsto z^{\beta-1}M(z)$
is bounded. 
\item \label{enu:ass_n11}If for all $z\in B_{\mathbb{C}}(r,r)$, we have
$\left(N_{11}(z)\right)^{-1}\in L(H_{1})$ then there is $c_{2}\in\mathbb{R}$
such that 
\[
\Re\left\langle \left(\left(z^{*}\right)^{\beta}N_{11}(z)\right)^{-1}\psi,\psi\right\rangle \geq c_{2}\langle\psi,\psi\rangle
\]
for all $\psi\in H_{1}.$ 
\end{enumerate}
\end{enumerate}
Some consequences of the latter assumptions are in order.
\begin{lem}
\label{lem:dmd}Assume that condition \eqref{enu:ass_m} holds. Then
$D(\partial_{0}M(\partial_{0}^{-1}))=H_{\nu}^{\beta}(\mathbb{R},H_{0})\coloneqq D(\partial_{0}^{\beta})$.\end{lem}
\begin{proof}
We first show $\partial_{0}M(\partial_{0}^{-1})=\overline{\partial_{0}^{1-\beta}M(\partial_{0}^{-1})\partial_{0}^{\beta}}.$
Since $\partial_{0}^{1-\beta}M(\partial_{0}^{-1})\partial_{0}^{\beta}\subseteq\partial_{0}M(\partial_{0}^{-1})$
and \foreignlanguage{english}{$\partial_{0}M(\partial_{0}^{-1})$}
is closed as a product of a bounded and a closed operator, we get
\[
\overline{\partial_{0}^{1-\beta}M(\partial_{0}^{-1})\partial_{0}^{\beta}}\subseteq\partial_{0}M(\partial_{0}^{-1}).
\]
Let now $u\in D(\partial_{0}M(\partial_{0}^{-1}))$ and set $u_{\varepsilon}\coloneqq(1+\varepsilon\partial_{0})^{-1}u\in D(\partial_{0})\subseteq D(\partial_{0}^{\beta})$
for $\varepsilon>0.$ Then $u_{\varepsilon}\to u$ in $L_{\nu}^{2}(\mathbb{R},H_{0})$
as $\varepsilon\to0$ and 
\begin{align*}
\partial_{0}^{1-\beta}M(\partial_{0}^{-1})\partial_{0}^{\beta}u_{\varepsilon} & =\partial_{0}M(\partial_{0}^{-1})u_{\varepsilon}\\
 & =(1+\varepsilon\partial_{0})^{-1}\partial_{0}M(\partial_{0}^{-1})u\to\partial_{0}M(\partial_{0}^{-1})u\quad(\varepsilon\to0).
\end{align*}
Thus, $u\in D\left(\overline{\partial_{0}^{1-\beta}M(\partial_{0}^{-1})\partial_{0}^{\beta}}\right)$
with $\overline{\partial_{0}^{1-\beta}M(\partial_{0}^{-1})\partial_{0}^{\beta}}u=\partial_{0}M(\partial_{0}^{-1})u,$
which proves the asserted equality. Now, by condition \eqref{enu:ass_m},
the operator $\partial_{0}^{1-\beta}M(\partial_{0}^{-1})$ is boundedly
invertible on $L_{\nu}^{2}(\mathbb{R},H)$ and hence, the operator
$\partial_{0}^{1-\beta}M(\partial_{0}^{-1})\partial_{0}^{\beta}$
is closed. The latter yields $\partial_{0}^{1-\beta}M(\partial_{0}^{-1})\partial_{0}^{\beta}=\partial_{0}M(\partial_{0}^{-1})$.
But $\partial_{0}^{1-\beta}M(\partial_{0}^{-1})$ is a bounded operator,
since $z\mapsto z^{\beta-1}M(z)$ is bounded by condition \eqref{enu:ass_m}.
Hence, $D\left(\partial_{0}^{1-\beta}M(\partial_{0}^{-1})\partial_{0}^{\beta}\right)=H_{\nu}^{\beta}(\mathbb{R},H_{0})$
and the assertion follows.\end{proof}
\begin{lem}
\label{lem:n_11inv}Assume condition \eqref{enu:ass_m_n}. Then for
all $z\in B_{\mathbb{C}}(r,r)$, the operator $N_{11}(z)$ is continuously
invertible.\end{lem}
\begin{proof}
The claim is immediate by putting $(\phi,\psi)=(0,\psi)$ in the positivity
estimate in condition \eqref{enu:ass_m_n}. \end{proof}
\begin{lem}
\label{lem:wt}Let $u\in L_{\nu}^{2}(\mathbb{R},H_{0})$. Then $u\in D(\partial_{0}^{\beta})$
if and only if $\sup_{\epsilon>0}\|\partial_{0}^{\beta}(1+\epsilon\partial_{0})^{-1}u\|<\infty.$ \end{lem}
\begin{proof}
From $\|(1+\epsilon\partial_{0})^{-1}\|\leq1$ for all $\epsilon>0$,
it follows that $u\in D(\partial_{0}^{\beta})$ is sufficient for
the supremum being finite. On the other hand, assume that the supremum
is finite. Then there is a sequence $(\epsilon_{n})_{n}$ in $(0,\infty)$
such that $\epsilon_{n}\to0$ as $n\to\infty$ and $v\coloneqq\lim_{n\to\infty}\partial_{0}^{\beta}(1+\epsilon_{n}\partial_{0})^{-1}u$
exists in the weak topology of $L_{\nu}^{2}(\mathbb{R},H_{0})$. By
the (weak) closedness of $\partial_{0}^{\beta}$ and the fact that
$(1+\epsilon_{n}\partial_{0})^{-1}u\to u$ as $n\to\infty$, we infer
$u\in D(\partial_{0}^{\beta})$.\end{proof}
\begin{thm}
\label{thm:maxreck} Assume conditions \eqref{enu:ass_m_n}, \eqref{enu:ass_m},
and \eqref{enu:ass_n11}. Then, for each $\nu>\frac{1}{2r},$ $B$
given in \eqref{eq:prob} is continuously invertible on $L_{\nu}^{2}(\mathbb{R},H_{0}\oplus H_{1})$
and for $(f,g)\in L_{\nu}^{2}(\mathbb{R},H_{0})\oplus H_{\nu}^{\beta}(\mathbb{R},H_{1}),$
we have $\overline{B}^{-1}(f,g)\in\left(H_{\nu}^{\beta}(\mathbb{R},H_{0})\oplus L_{\nu}^{2}(\mathbb{R},H_{1})\right)\cap\left(D\left(\left(\begin{array}{cc}
0 & -C^{*}\\
C & 0
\end{array}\right)\right)\right)$.\end{thm}
\begin{proof}
We want to apply Corollary \ref{cor:mrst}. For this, we have to show
that 
\[
(u,v)\coloneqq\overline{B}^{-1}(f,g)\in D\left(\left(\begin{array}{cc}
\partial_{0}M(\partial_{0}^{-1}) & 0\\
0 & 0
\end{array}\right)+\left(\begin{array}{cc}
N_{00}(\partial_{0}^{-1}) & N_{01}(\partial_{0}^{-1})\\
N_{10}(\partial_{0}^{-1}) & N_{11}(\partial_{0}^{-1})
\end{array}\right)\right).
\]
By the boundedness of $\left(\begin{array}{cc}
N_{00}(\partial_{0}^{-1}) & N_{01}(\partial_{0}^{-1})\\
N_{10}(\partial_{0}^{-1}) & N_{11}(\partial_{0}^{-1})
\end{array}\right)$, we are left with showing that $u\in D(\partial_{0}M(\partial_{0}^{-1}))$.
By Lemma \ref{lem:dmd}, we need to show that $u\in D(\partial_{0}^{\beta})$.
Invoking Lemma \ref{lem:wt}, it suffices to show that 
\[
\sup_{\epsilon>0}\|\partial_{0}^{\beta}(1+\epsilon\partial_{0})^{-1}u\|<\infty.
\]
So, let $\epsilon>0$ and define $u_{\epsilon}\coloneqq(1+\epsilon\partial_{0})^{-1}u.$
We further set $v_{\epsilon}\coloneqq(1+\epsilon\partial_{0})^{-1}v$.
By Theorem \ref{thm:st} (note that $\left(\begin{array}{cc}
0 & -C^{*}\\
C & 0
\end{array}\right)$ is skew-selfadjoint; and that the needed inequality for the application
of Theorem \ref{thm:st} is warranted by \eqref{enu:ass_m_n}), we
have that 
\[
(u_{\epsilon},v_{\epsilon})\in D(B)=D\left(\left(\begin{array}{cc}
\partial_{0}M(\partial_{0}^{-1}) & 0\\
0 & 0
\end{array}\right)\right)\cap D\left(\left(\begin{array}{cc}
0 & -C^{*}\\
C & 0
\end{array}\right)\right).
\]
Thus, we read off $v_{\epsilon}\in D(C^{*})$ as well as $u_{\epsilon}\in D(C)\cap D(\partial_{0}M(\partial_{0}^{-1})).$
Moreover, we have the equalities 
\begin{align*}
\partial_{0}M(\partial_{0}^{-1})u_{\epsilon}+N_{00}(\partial_{0}^{-1})u_{\epsilon}+N_{01}(\partial_{0}^{-1})v_{\epsilon}-C^{*}v_{\epsilon} & =f_{\epsilon},\\
N_{11}(\partial_{0}^{-1})v_{\epsilon}+N_{10}(\partial_{0}^{-1})u_{\epsilon}+Cu_{\epsilon} & =g_{\epsilon},
\end{align*}
where $f_{\varepsilon}\coloneqq(1+\varepsilon\partial_{0})^{-1}f$
and $g_{\varepsilon}\coloneqq(1+\varepsilon\partial_{0})^{-1}g.$
Next, letting $\epsilon\to0$ in the second equality, we infer by
the closedness of $C$ that $u\in D(C)$ and 
\[
N_{11}(\partial_{0}^{-1})v+N_{10}(\partial_{0}^{-1})u+Cu=g.
\]
Furthermore, we get 
\begin{multline}
\|Cu\|\leq\|g\|+\|N_{11}(\partial_{0}^{-1})\|\|v\|+\|N_{10}(\partial_{0}^{-1})\|\|u\|\\
\leq\left(1+\frac{1}{c}\left(\|N_{11}(\partial_{0}^{-1})\|+\|N_{10}(\partial_{0}^{-1})\|\right)\right)\left(\|g\|+\|f\|\right),\label{eq:Cu}
\end{multline}
where we have used condition \eqref{enu:ass_m_n}. By Lemma \ref{lem:n_11inv},
we also get 
\[
v_{\epsilon}=\left(N_{11}(\partial_{0}^{-1})\right)^{-1}(-N_{10}(\partial_{0}^{-1})u_{\epsilon}-Cu_{\epsilon}+g_{\epsilon}).
\]
Substituting the latter equation into the first one, we arrive at
\begin{multline*}
\partial_{0}M(\partial_{0}^{-1})u_{\epsilon}+N_{00}(\partial_{0}^{-1})u_{\epsilon}+N_{01}(\partial_{0}^{-1})\left(N_{11}(\partial_{0}^{-1})\right)^{-1}(-N_{10}(\partial_{0}^{-1})u_{\epsilon}-Cu_{\epsilon}+g_{\epsilon})\\
-C^{*}\left(N_{11}(\partial_{0}^{-1})\right)^{-1}(-N_{10}(\partial_{0}^{-1})u_{\epsilon}-Cu_{\epsilon}+g_{\epsilon})=f_{\epsilon}.
\end{multline*}
Hence, 
\begin{align*}
\partial_{0}M(\partial_{0}^{-1})u_{\epsilon} & =-N_{00}(\partial_{0}^{-1})u_{\epsilon}+N_{01}(\partial_{0}^{-1})\left(N_{11}(\partial_{0}^{-1})\right)^{-1}N_{10}(\partial_{0}^{-1})u_{\epsilon}\\
 & \quad+N_{01}(\partial_{0}^{-1})\left(N_{11}(\partial_{0}^{-1})\right)^{-1}Cu_{\epsilon}-N_{01}(\partial_{0}^{-1})\left(N_{11}(\partial_{0}^{-1})\right)^{-1}g_{\epsilon}\\
 & \quad-C^{*}\left(N_{11}(\partial_{0}^{-1})\right)^{-1}\left(N_{10}(\partial_{0}^{-1})u_{\epsilon}+Cu_{\epsilon}-g_{\epsilon}\right)+f_{\epsilon}.
\end{align*}
We apply $\langle\cdot,\partial_{0}^{\beta}u_{\epsilon}\rangle_{L_{\nu}^{2}}$
to the latter equation, take real parts and use condition \eqref{enu:ass_m}
to get 
\begin{align*}
c_{1}\Re\langle\partial_{0}^{\beta}u_{\epsilon},\partial_{0}^{\beta}u_{\epsilon}\rangle & \leq\Re\langle\partial_{0}^{1-\beta}M(\partial_{0}^{-1})\partial_{0}^{\beta}u_{\epsilon},\partial_{0}^{\beta}u_{\epsilon}\rangle\\
 & =\Re\langle\partial_{0}M(\partial_{0}^{-1})u_{\epsilon},\partial_{0}^{\beta}u_{\epsilon}\rangle\\
 & =\Re\left\langle -N_{00}(\partial_{0}^{-1})u_{\epsilon}+N_{01}(\partial_{0}^{-1})\left(N_{11}(\partial_{0}^{-1})\right)^{-1}N_{10}(\partial_{0}^{-1})u_{\epsilon},\partial_{0}^{\beta}u_{\epsilon}\right\rangle \\
 & \quad+\Re\left\langle N_{01}(\partial_{0}^{-1})\left(N_{11}(\partial_{0}^{-1})\right)^{-1}Cu_{\epsilon}-N_{01}(\partial_{0}^{-1})\left(N_{11}(\partial_{0}^{-1})\right)^{-1}g_{\epsilon},\partial_{0}^{\beta}u_{\epsilon}\right\rangle \\
 & \quad+\Re\left\langle -C^{*}\left(N_{11}(\partial_{0}^{-1})\right)^{-1}\left(N_{10}(\partial_{0}^{-1})u_{\epsilon}+Cu_{\epsilon}-g_{\epsilon}\right),\partial_{0}^{\beta}u_{\epsilon}\right\rangle \\
 & \quad+\Re\left\langle f_{\epsilon},\partial_{0}^{\beta}u_{\epsilon}\right\rangle .
\end{align*}
We recall that $u\in D(C)$ and, hence, $u_{\epsilon}\in D(C)$ as
well as $\partial_{0}^{\beta}u_{\epsilon}\in D(C)$. Thus, we have
\begin{align*}
c_{1}\Re\langle\partial_{0}^{\beta}u_{\epsilon},\partial_{0}^{\beta}u_{\epsilon}\rangle & \leq\Re\left\langle -N_{00}(\partial_{0}^{-1})u_{\epsilon}+N_{01}(\partial_{0}^{-1})\left(N_{11}(\partial_{0}^{-1})\right)^{-1}N_{10}(\partial_{0}^{-1})u_{\epsilon},\partial_{0}^{\beta}u_{\epsilon}\right\rangle \\
 & \quad+\Re\left\langle N_{01}(\partial_{0}^{-1})\left(N_{11}(\partial_{0}^{-1})\right)^{-1}Cu_{\epsilon}-N_{01}(\partial_{0}^{-1})\left(N_{11}(\partial_{0}^{-1})\right)^{-1}g_{\epsilon},\partial_{0}^{\beta}u_{\epsilon}\right\rangle \\
 & \quad-\Re\left\langle \left(\partial_{0}^{\beta}\right)^{*}\left(N_{11}(\partial_{0}^{-1})\right)^{-1}\left(N_{10}(\partial_{0}^{-1})u_{\epsilon}+Cu_{\epsilon}-g_{\epsilon}\right),Cu_{\epsilon}\right\rangle \\
 & \quad+\Re\left\langle f_{\epsilon},\partial_{0}^{\beta}u_{\epsilon}\right\rangle .
\end{align*}
We note that apart from the term $\Re\left\langle \left(\partial_{0}^{\beta}\right)^{*}\left(N_{11}(\partial_{0}^{-1})\right)^{-1}\left(N_{10}(\partial_{0}^{-1})u_{\epsilon}+Cu_{\epsilon}-g_{\epsilon}\right),Cu_{\epsilon}\right\rangle $
the remaining terms of the right-hand side can by estimated by 
\[
K_{1}\|\partial_{0}^{\beta}u_{\epsilon}\|
\]
for some constant $K_{1}\geq0$, where we also used \eqref{eq:Cu}
as well as $\|(1+\epsilon\partial_{0})^{-1}\|\leq1$. For the treatise
of $\Re\left\langle \left(\partial_{0}^{\beta}\right)^{*}\left(N_{11}(\partial_{0}^{-1})\right)^{-1}\left(N_{10}(\partial_{0}^{-1})u_{\epsilon}+Cu_{\epsilon}-g_{\epsilon}\right),Cu_{\epsilon}\right\rangle $
we estimate with the help of condition \eqref{enu:ass_n11} (note
that the implication is not void by Lemma \ref{lem:n_11inv}) 
\begin{align*}
 & -\Re\left\langle \left(\partial_{0}^{\beta}\right)^{*}\left(N_{11}(\partial_{0}^{-1})\right)^{-1}\left(N_{10}(\partial_{0}^{-1})u_{\epsilon}+Cu_{\epsilon}-g_{\epsilon}\right),Cu_{\epsilon}\right\rangle \\
 & =-\Re\left\langle \left(\partial_{0}^{\beta}\right)^{*}\left(N_{11}(\partial_{0}^{-1})\right)^{-1}N_{10}(\partial_{0}^{-1})u_{\epsilon},Cu_{\epsilon}\right\rangle \\
 & \quad-\Re\left\langle \left(\partial_{0}^{\beta}\right)^{*}\left(N_{11}(\partial_{0}^{-1})\right)^{-1}Cu_{\epsilon}-\left(\partial_{0}^{\beta}\right)^{*}\left(N_{11}(\partial_{0}^{-1})\right)^{-1}g_{\epsilon},Cu_{\epsilon}\right\rangle \\
 & =-\Re\left\langle \left(N_{11}(\partial_{0}^{-1})\right)^{-1}N_{10}(\partial_{0}^{-1})\left(\partial_{0}^{\beta}\right)^{*}u_{\epsilon},Cu_{\epsilon}\right\rangle \\
 & \quad-\Re\left\langle \left(\partial_{0}^{\beta}\right)^{*}\left(N_{11}(\partial_{0}^{-1})\right)^{-1}Cu_{\epsilon},Cu_{\epsilon}\right\rangle +\Re\left\langle \left(N_{11}(\partial_{0}^{-1})\right)^{-1}\left(\partial_{0}^{\beta}\right)^{*}g_{\epsilon},Cu_{\epsilon}\right\rangle \\
 & \leq\left\Vert \left(N_{11}(\partial_{0}^{-1})\right)^{-1}N_{10}(\partial_{0}^{-1})\right\Vert \left\Vert \left(\partial_{0}^{\beta}\right)^{*}u_{\epsilon}\right\Vert \left\Vert Cu\right\Vert +|c_{1}|\left\Vert Cu\right\Vert ^{2}\\
 & \quad+\left\Vert \left(N_{11}(\partial_{0}^{-1})\right)^{-1}\right\Vert \left\Vert \left(\partial_{0}^{\beta}\right)^{*}g_{\epsilon}\right\Vert \left\Vert Cu\right\Vert \\
 & =\left\Vert \left(N_{11}(\partial_{0}^{-1})\right)^{-1}N_{10}(\partial_{0}^{-1})\right\Vert \left\Vert \partial_{0}^{\beta}u_{\epsilon}\right\Vert \left\Vert Cu\right\Vert +|c_{1}|\left\Vert Cu\right\Vert ^{2}\\
 & \quad+\left\Vert \left(N_{11}(\partial_{0}^{-1})\right)^{-1}\right\Vert \left\Vert \partial_{0}^{\beta}g\right\Vert \left\Vert Cu\right\Vert \\
 & \leq K_{2}\left\Vert \partial_{0}u_{\epsilon}\right\Vert +K_{3}
\end{align*}
for some $K_{2},K_{3}\geq0$, where we have again used \eqref{eq:Cu}.
Hence, we get for $p\coloneqq\left(K_{1}+K_{2}\right)/c\geq0$ and
$q\coloneqq K_{3}/c\geq0$ that 
\[
\left\Vert \partial_{0}^{\beta}u_{\epsilon}\right\Vert ^{2}\leq p\left\Vert \partial_{0}^{\beta}u_{\epsilon}\right\Vert +q,
\]
which implies 
\[
\left\Vert \partial_{0}^{\beta}u_{\epsilon}\right\Vert \leq\frac{p}{2}+\sqrt{\frac{p^{2}}{4}+q}.
\]
Thus, $u\in D(\partial_{0}^{\beta}),$ by Lemma \ref{lem:wt}. 
\end{proof}
Another, perhaps more familiar looking, maximal regularity result
can now be deduced from Theorem \ref{thm:maxreck}:
\begin{cor}
\label{cor:MR2nd}Assume conditions \eqref{enu:ass_m_n},\eqref{enu:ass_m}
and \eqref{enu:ass_n11} to be satisfied, $\nu>1/(2r)$. Then, for
all $f\in L_{\nu}^{2}(\mathbb{R},H_{0})$, there exists a unique 
\begin{eqnarray*}
u & \in & H_{\nu}^{\beta}(\mathbb{R},H_{0})\cap D\left(C^{*}N_{11}(\partial_{0}^{-1})^{-1}\left(C+N_{10}(\partial_{0}^{-1})\right)\right)
\end{eqnarray*}
satisfying 
\begin{multline}
\partial_{0}M(\partial_{0}^{-1})u+N_{00}\left(\partial_{0}^{-1}\right)u-N_{01}\left(\partial_{0}^{-1}\right)\left(N_{11}\left(\partial_{0}^{-1}\right)\right)^{-1}\left(C+N_{10}\left(\partial_{0}^{-1}\right)\right)u\\
+C^{*}\left(N_{11}\left(\partial_{0}^{-1}\right)\right)^{-1}\left(C+N_{10}\left(\partial_{0}^{-1}\right)\right)u=f.\label{eq:2ndo}
\end{multline}
\end{cor}
\begin{proof}
Using condition \eqref{enu:ass_m_n}, by Theorem \ref{thm:st}, we
infer the existence of a unique $\left(v,w\right)\in L_{\nu}^{2}(\mathbb{R},H_{0}\oplus H_{1})$
such that 
\[
\overline{\left(\partial_{0}\left(\begin{array}{cc}
M(\partial_{0}^{-1}) & 0\\
0 & 0
\end{array}\right)+\left(\begin{array}{cc}
N_{00}(\partial_{0}^{-1}) & N_{01}(\partial_{0}^{-1})\\
N_{10}(\partial_{0}^{-1}) & N_{11}(\partial_{0}^{-1})
\end{array}\right)+\left(\begin{array}{cc}
0 & -C^{*}\\
C & 0
\end{array}\right)\right)}\left(\begin{array}{c}
v\\
w
\end{array}\right)=\left(\begin{array}{c}
f\\
0
\end{array}\right).
\]
By Theorem \ref{thm:maxreck} (and Lemma \ref{lem:dmd}), we get 
\begin{align*}
\left(\begin{array}{c}
v\\
w
\end{array}\right) & \in H_{\nu}^{\beta}(\mathbb{R},H_{0})\oplus L_{\nu}^{2}(\mathbb{R},H_{1})\cap D\left(\left(\begin{array}{cc}
0 & -C^{*}\\
C & 0
\end{array}\right)\right)\\
 & \ =D\left(\partial_{0}\left(\begin{array}{cc}
M(\partial_{0}^{-1}) & 0\\
0 & 0
\end{array}\right)\right)\cap D\left(\left(\begin{array}{cc}
0 & -C^{*}\\
C & 0
\end{array}\right)\right)
\end{align*}
Hence, 
\begin{align}
\left(\begin{array}{c}
f\\
0
\end{array}\right) & =\overline{\left(\partial_{0}\left(\begin{array}{cc}
M(\partial_{0}^{-1}) & 0\\
0 & 0
\end{array}\right)+\left(\begin{array}{cc}
N_{00}(\partial_{0}^{-1}) & N_{01}(\partial_{0}^{-1})\\
N_{10}(\partial_{0}^{-1}) & N_{11}(\partial_{0}^{-1})
\end{array}\right)+\left(\begin{array}{cc}
0 & -C^{*}\\
C & 0
\end{array}\right)\right)}\left(\begin{array}{c}
v\\
w
\end{array}\right)\nonumber \\
 & =\left(\partial_{0}\left(\begin{array}{cc}
M(\partial_{0}^{-1}) & 0\\
0 & 0
\end{array}\right)+\left(\begin{array}{cc}
N_{00}(\partial_{0}^{-1}) & N_{01}(\partial_{0}^{-1})\\
N_{10}(\partial_{0}^{-1}) & N_{11}(\partial_{0}^{-1})
\end{array}\right)+\left(\begin{array}{cc}
0 & -C^{*}\\
C & 0
\end{array}\right)\right)\left(\begin{array}{c}
v\\
w
\end{array}\right)\nonumber \\
 & =\left(\begin{array}{c}
\partial_{0}M(\partial_{0}^{-1})v+N_{00}(\partial_{0}^{-1})v+N_{01}(\partial_{0}^{-1})w-C^{*}w\\
N_{10}(\partial_{0}^{-1})v+N_{11}(\partial_{0}^{-1})w+Cv
\end{array}\right).\label{eq:syst}
\end{align}
With Lemma \ref{lem:n_11inv}, we obtain from the second line 
\[
w=-\left(N_{11}(\partial_{0}^{-1})\right)^{-1}\left(C+N_{10}(\partial_{0}^{-1})\right)v.
\]
Substituting the latter equation into the first equation of \eqref{eq:syst},
we obtain \eqref{eq:2ndo}. On the other hand, given $u\in H_{\nu}^{\beta}(\mathbb{R},H_{0})\cap D\left(C^{*}N_{11}(\partial_{0}^{-1})^{-1}\left(C+N_{10}(\partial_{0}^{-1})\right)\right)$
satisfying \eqref{eq:2ndo}, we deduce that 
\[
\left(\begin{array}{c}
u\\
-\left(N_{11}(\partial_{0}^{-1})\right)^{-1}\left(C+N_{10}(\partial_{0}^{-1})\right)u
\end{array}\right)
\]
is a solution of \eqref{eq:syst}, the solution of which being unique.
Thus, the uniqueness statement is also settled. 
\end{proof}

\section{Some Examples}

Although the strength of the above result lies in the generalty of
the ``material laws'' accessible, the approach is perhaps best illustrated
and by making a link to known results obtained by a different approach.
In this spirit, our first example deals with paradigm of maximal regularity,
the heat equation, to illustrate the different perspective of our
approach on this issue. We then continue with slightly more complex
example cases from the literature, which may not be seen to be covered
by the general approach developed here. This includes a concluding
example for a fractional-in-time evolutionary problem.

\subsection{The heat equation }

As a warm-up example we consider the paradigmatic case of the heat
transport. Let $\Omega\subseteq\mathbb{R}^{3}$ be a non-empty open
where the heat transport is supposed to take place. We consider the
equations of heat conduction in the body $\Omega$, which consists
of the balance of momentum law 
\[
\partial_{0}\theta+\dive q=f,
\]
where $\theta:\mathbb{R}\times\Omega\to\mathbb{C}$ denotes the temperature
density, $q:\mathbb{R}\times\Omega\to\mathbb{C}^{3}$ stands for the
heat flux and $f:\mathbb{R}\times\Omega\to\mathbb{C}$ is an external
heat source forcing term, and Fourier's law 
\[
q=-k\grad\theta,
\]
where $k\in L({L^{2}(\Omega)^{3}},{L^{2}(\Omega)^{3}})$ is a bounded
selfadjoint operator satisfying 
\[
\Re\langle k\psi,\psi\rangle_{{L^{2}(\Omega)^{3}}}\geq c\langle\psi,\psi\rangle_{{L^{2}(\Omega)^{3}}}\quad(\psi\in{L^{2}(\Omega)^{3}})
\]
for some $c>0$, modeling the heat conductivity of the medium occupying
$\Omega$. If we impose suitable boundary conditions, say -- a homogeneous
Dirichlet boundary condition, on $\theta$, we end up with the following
system 
\begin{equation}
\left(\partial_{0}\left(\begin{array}{cc}
1 & 0\\
0 & 0
\end{array}\right)+\left(\begin{array}{cc}
0 & 0\\
0 & k^{-1}
\end{array}\right)+\left(\begin{array}{cc}
0 & \dive\\
\grad_{0} & 0
\end{array}\right)\right)\left(\begin{array}{c}
\theta\\
q
\end{array}\right)=\left(\begin{array}{c}
f\\
0
\end{array}\right),\label{eq:heat}
\end{equation}
where $\grad_{0}$ is defined as the distributional gradient with
domain $H_{0}^{1}(\Omega)$ and $\dive\coloneqq-\grad_{0}^{\ast}.$
Thus, we are indeed in the setting studied in the previous section.
Since $k$ is bounded, selfadjoint and strictly positive definite,
so is $k^{-1}.$ Thus, conditions \eqref{enu:ass_m} (with $\beta=1$)
and \eqref{enu:ass_m_n} are clearly satisfied. Moreover, for $z\in B_{\mathbb{C}}(r,r)$
we have $z^{-1}=\i t+\rho$ for some $\rho>\frac{1}{2r},t\in\mathbb{R}$
and hence, 
\[
\Re\langle\left(z^{\ast}k^{-1}\right)^{-1}\psi,\psi\rangle=\Re(-\ii t+\rho)\langle k\psi,\psi\rangle\geq\rho c\langle\psi,\psi\rangle
\]
for each $\psi\in{L^{2}(\Omega)^{3}},$ where we have used the selfadjointness
of $k$. This proves that condition \eqref{enu:ass_n11} is satisfied
and thus, Theorem \ref{thm:maxreck} yields maximal regularity of
\eqref{eq:heat}. In view of Corollary \ref{cor:MR2nd}, we end up
with the following result:
\begin{cor}
For all $\nu>0,$ $f\in L_{\nu}^{2}(\mathbb{R},L^{2}(\Omega))$, there
exists a unique $u\in H_{\nu}^{1}(\mathbb{R},{L^{2}(\Omega)^{3}})\cap D(\dive k\grad_{0})$
such that 
\[
\partial_{0}u-\dive k\grad_{0}u=f.
\]
\end{cor}
\begin{rem}
We emphasize that each boundary condition yielding an operator matrix
of the form $\left(\begin{array}{cc}
0 & -C^{\ast}\\
C & 0
\end{array}\right)$ allows the application of Theorem \ref{thm:maxreck}. For several
examples of such boundary condition, including mixed and non-local
ones we refer to \cite{Picard2016_graddiv}. 
\end{rem}

\subsection{A second order equation}

Following \cite[Example 6.1]{Batty2015}, where the much deeper issue
of maximal regularity in certain interpolation spaces is addressed,
we consider the equation 
\[
\partial_{0}^{2}\theta+C^{*}\left(A+B\partial_{0}\right)C\theta=f,
\]
where $C:D(C)\subseteq H_{0}\to H_{1}$ is densely defined closed
and linear between the two Hilbert spaces $H_{0}$ and $H_{1}$, and
$B\in L(H_{1})$ is selfadjoint, strictly positive definite and $A\in L(H_{1})$.
First, we note that for $\rho>0$ large enough the operator $(A+B\partial_{0})=\partial_{0}\left(A\partial_{0}^{-1}+B\right)=\partial_{0}B\left(B^{-1}A\partial_{0}^{-1}+1\right)$
is continuously invertible on $L_{\rho}^{2}(\mathbb{R},H_{1}),$ due
to a Neumann series argument (for this recall that $\|\partial_{0}^{-1}\|\leq1/\rho$).
Hence, setting $w\coloneqq\partial_{0}\theta,q\coloneqq-(A+B\partial_{0})C\theta$,
we may rewrite the above problem as a first order equation of the
form 
\[
\left(\partial_{0}\left(\begin{array}{cc}
1 & 0\\
0 & 0
\end{array}\right)+\left(\begin{array}{cc}
0 & 0\\
0 & \left(B^{-1}A\partial_{0}^{-1}+1\right)^{-1}B^{-1}
\end{array}\right)+\left(\begin{array}{cc}
0 & -C^{*}\\
C & 0
\end{array}\right)\right)\left(\begin{array}{c}
w\\
q
\end{array}\right)=\left(\begin{array}{c}
f\\
0
\end{array}\right).
\]

Thus, Theorem \ref{thm:maxreck} is applicable with the choices 
\[
M(z)=1,\quad N(z)=\left(\begin{array}{cc}
0 & 0\\
0 & (Az+B)^{-1}
\end{array}\right),\quad g=0.
\]
Indeed, condition \eqref{enu:ass_m} (for $\beta=1$) is obviously
satisfied while condition \eqref{enu:ass_m_n} follows from 
\begin{align*}
\Re\langle N_{11}(z)\psi,\psi\rangle_{H_{1}} & =\Re\langle B^{-1}\psi,\psi\rangle-\Re z\langle B^{-1}AB^{-1}(AB^{-1}z+1)^{-1}\psi,\psi\rangle_{H_{1}}\\
 & \geq c\langle\psi,\psi\rangle-\frac{\|B^{-1}AB^{-1}\|}{\frac{1}{|z|}-\|B^{-1}A\|}\langle\psi,\psi\rangle_{H_{1}}\\
 & \geq\left(c-\frac{\|B^{-1}AB^{-1}\|}{\frac{1}{r}-\|B^{-1}A\|}\right)\langle\psi,\psi\rangle_{H_{1}}\quad(\psi\in H_{1})
\end{align*}
for $z\in B_{\mathbb{C}}(r,r)$ with $r>0$ small enough (which corresponds
to $\rho>0$ large enough in the argumentation above), where $c>0$
is a positive definiteness constant of $B^{-1}$, that is, $B^{-1}\geq c$.
For showing condition \eqref{enu:ass_n11} (for $\beta=1$), we compute
\[
\Re\langle\left(z^{\ast}\right)^{-1}(Az+B)\psi,\psi\rangle_{H_{1}}\geq-\|A\|\langle\psi,\psi\rangle+\frac{1}{2r}c'\langle\psi,\psi\rangle,
\]
for each $z\in B_{\mathbb{C}}(r,r)$, where we have used the selfadjointness
of $B$ and that $B\geq c'$ for some $c'>0$ by assumption. The corresponding
statement for the equation, we originally started out with is as follows.
\begin{cor}
There exists $\nu_{0}>0$ such that for all $\nu\geq\nu_{0}$ the
following holds: For all $f\in L_{\nu}^{2}(\mathbb{R},H_{0})$ there
exists a unique $\theta\in H_{\nu}^{2}(\mathbb{R},H_{0})\cap D\left(C^{*}\left(A+B\partial_{0}\right)C\right)$
satisfying 
\[
\partial_{0}^{2}\theta+C^{*}\left(A+B\partial_{0}\right)C\theta=f.
\]
\end{cor}
\begin{proof}
Again, we rely on Corollary \ref{cor:MR2nd} for $\beta=1$. Note
that in the above computations, we used the substitution $w=\partial_{0}\theta.$
We infer that $w\in H_{\nu}^{1}(\mathbb{R},H_{0})$, which yields
$\theta\in H_{\nu}^{2}(\mathbb{R},H_{0})$. 
\end{proof}

\subsection{A second order integro-differential equation}

Let $C:D(C)\subseteq H_{0}\to H_{1}$ densely defined closed and linear,
$k:\mathbb{R}_{\geq0}\to L(H_{1})$ weakly measurable, such that $t\mapsto\|k(t)\|$
is measurable and $|k|_{L_{\rho_{0}}^{1}}\coloneqq\intop_{0}^{\infty}\|k(t)\|e^{-\rho_{0}t}\mbox{ d}t<\infty$
for some $\rho_{0}>0$. Moreover, let $A,B\in L(H_{1})$ with $A$
selfadjoint and strictly positive definite. We consider the following
equation 
\begin{equation}
\left(\partial_{0}^{2}+C^{\ast}\left(\partial_{0}A+B+k\ast\right)C\right)u=f,\label{eq:integro_orig}
\end{equation}
where the convolution operator $k\ast$ is defined by 
\begin{align*}
k\ast:L_{\rho}^{2}(\mathbb{R},H_{1}) & \to L_{\rho}^{2}(\mathbb{R},H_{1})\\
g & \mapsto\left(t\mapsto\intop_{0}^{\infty}k(s)g(t-s)\mbox{ d}s\right)
\end{align*}
for $\rho\geq\rho_{0}.$ By Young's inequality we have that 
\[
\|k\ast\|_{L(L_{\rho}^{2}(\mathbb{R},H_{1}))}\leq|k|_{L_{\rho}^{1}}\leq|k|_{L_{\rho_{0}}^{1}}<\infty,
\]
so that $k\ast$ is a bounded linear operator on $L_{\rho}^{2}(\mathbb{R},H_{1})$
for each $\rho\geq\rho_{0}.$ Moreover, by monotone convergence, we
get that $\limsup_{\rho\to\infty}\|k\ast\|_{L(L_{\rho}^{2}(\mathbb{R},H_{1}))}\leq\lim_{\rho\to\infty}|k|_{L_{\rho}^{1}}=0$.
For a treatment of integro-differential equations within the framework
of evolutionary problems we refer to \cite{Trostorff2012_integro},
where this is a special case in the discussion of problems with monotone
relations. We rewrite the above problem as a first order problem in
the new unknowns $v\coloneqq\partial_{0}u$ and $q\coloneqq-\left(A+\partial_{0}^{-1}\left(B+k\ast\right)\right)Cv.$
Thus, we arrive at 
\begin{equation}
\left(\partial_{0}\left(\begin{array}{cc}
1 & 0\\
0 & 0
\end{array}\right)+\left(\begin{array}{cc}
0 & 0\\
0 & \left(A+\partial_{0}^{-1}(B+k\ast)\right)^{-1}
\end{array}\right)+\left(\begin{array}{cc}
0 & -C^{\ast}\\
C & 0
\end{array}\right)\right)\left(\begin{array}{c}
v\\
q
\end{array}\right)=\left(\begin{array}{c}
f\\
0
\end{array}\right).\label{eq:integro}
\end{equation}
We note that the operator $A+\partial_{0}^{-1}(B+k\ast)$ is indeed
boundedly invertible on $L_{\rho}^{2}(\mathbb{R},H_{1})$ for sufficiently
large $\rho>0,$ since 
\[
A+\partial_{0}^{-1}(B+k\ast)=A\left(1+\partial_{0}^{-1}A^{-1}\left(B+k\ast\right)\right)
\]
and 
\[
\|\partial_{0}^{-1}A^{-1}\left(B+k\ast\right)\|_{L(L_{\rho}^{2}(\mathbb{R},H_{1}))}\leq\frac{1}{\rho}\|A^{-1}\|\left(\|B\|+|k|_{L_{\rho}^{1}}\right)<1
\]
for $\rho$ sufficiently large. Moreover, we note that the above problem
is an equation of the form discussed in Section \ref{sec:The-main-result}
with 
\[
N_{11}(z)\coloneqq\left(A+z\left(B+\sqrt{2\pi}\hat{k}(-\ii z^{-1})\right)\right)^{-1},
\]
where $\hat{k}$ denotes the Fourier-transform of $k$ (see \cite{Trostorff2012_integro}
for more details). Condition \eqref{enu:ass_m} (for $\beta=1$) is
obviously satisfied in this situation. Moreover, since 
\[
N_{11}(z)=A^{-1}+A^{-1}\sum_{k=1}^{\infty}(-z)^{k}\left(\left(B+\sqrt{2\pi}\:\hat{k}(-\ii z^{-1})\right)A^{-1}\right)^{k}
\]

by Neumann series expansion, we infer that $\Re N_{11}(z)$ is uniformly
strictly positive definite for $z\in B_{\mathbb{C}}(\frac{1}{2\rho},\frac{1}{2\rho})$
for $\rho>0$ large enough, since $A^{-1}$ is strictly positive definite
and 
\begin{align*}
 & \sup_{z\in B_{\mathbb{C}}(\frac{1}{2\rho},\frac{1}{2\rho})}\left\Vert A^{-1}\sum_{k=1}^{\infty}(-z)^{k}\left(\left(B+\sqrt{2\pi}\:\hat{k}(-\ii z^{-1})\right)A^{-1}\right)^{k}\right\Vert \\
 & \leq\sup_{z\in B_{\mathbb{C}}(\frac{1}{2\rho},\frac{1}{2\rho})}\|A^{-1}\|\left(\frac{|z|\left(\|B\|+|k|_{L_{\rho_{0}}^{1}}\right)\|A^{-1}\|}{1-|z|\left(\|B\|+|k|_{L_{\rho_{0}}^{1}}\right)\|A^{-1}\|}\right)\to0\quad(\rho\to\infty).
\end{align*}
This yields that condition \eqref{enu:ass_m_n} is also satisfied.
Finally, using the representation $z^{-1}=\ii t+\rho$ for some $t\in\mathbb{R},\rho>\rho_{0}$
large enough, we obtain 
\begin{align*}
\Re\left(z^{\ast}N_{11}(z)\right)^{-1} & =\Re\left(z^{\ast}\right)^{-1}A+\Re\frac{z}{z^{\ast}}\left(B+\sqrt{2\pi}\hat{k}(-\ii z^{-1})\right)\\
 & \geq\rho A-\left(\|B\|+|k|_{L_{\rho}^{1}}\right)\\
 & \geq\rho_{0}c-\left(\|B\|+|k|_{L_{\rho_{0}}^{1}}\right),
\end{align*}
with $c>0$ such that $A\geq c.$ This shows condition \eqref{enu:ass_n11}
($\beta=1$). Thus, Corollary \ref{cor:MR2nd} applies with $\beta=1$
and yields the maximal regularity of \eqref{eq:integro_orig}.
\begin{rem}
The maximal regularity of a similar problem as \eqref{eq:integro_orig}
was studied in \cite{Sforza1995} in a Banach space setting, where
the operators $A$ and $B$ were replaced by real scalars, the kernel
$k$ was assumed to be real-valued and the operator $C^{\ast}C$ was
replaced by a generator of an analytic semigroup. 
\end{rem}

\subsection{A partial differential equation of fractional type}

We conclude with the following example taken from \cite{Ponce2013},
where the maximal regularity of the equation 
\[
\partial_{0}^{\beta}u-(1+k*)Au=f
\]
has been addressed in spaces of (Banach space-valued) Hölder continuous
functions for some $\beta\in]0,1[$. Here $A$ is a sectorial operator
and $k$ is a suitable integrable, scalar-valued function, which is
supported in the positive reals only. As the case of convolutions
has been addressed in the previous two subsections, already, we focus
on the simplified equation 
\begin{equation}
\partial_{0}^{\beta}u+C^{*}Cu=f,\label{eq:frac}
\end{equation}
where $C\colon D(C)\subseteq H_{0}\to H_{1}$ is densely defined and
closed in the Hilbert spaces $H_{0}$ and $H_{1}$. We show that the
equation \eqref{eq:frac} admits maximal regularity in $L_{\rho}^{2}(\mathbb{R},H_{0})$
for all $\rho>0$. So, let $\rho>0$. Setting $q\coloneqq-Cu$, a
corresponding 2-by-2 block operator matrix formulation reads 
\[
\left(\partial_{0}\left(\begin{array}{cc}
\partial_{0}^{\beta-1} & 0\\
0 & 0
\end{array}\right)+\left(\begin{array}{cc}
0 & 0\\
0 & 1
\end{array}\right)+\left(\begin{array}{cc}
0 & -C^{*}\\
C & 0
\end{array}\right)\right)\left(\begin{array}{c}
u\\
q
\end{array}\right)=\left(\begin{array}{c}
f\\
0
\end{array}\right).
\]
We want to apply Theorem \ref{thm:maxreck} (or Corollary \ref{cor:MR2nd})
to 
\[
M(z)=z^{1-\beta},\quad N(z)=\left(\begin{array}{cc}
0 & 0\\
0 & 1
\end{array}\right).
\]
For this, note that condition \eqref{enu:ass_m} is satisfied, since
for all $r>1/(2\rho)$, we have 
\[
\Re\langle z^{\beta-1}M(z)\phi,\phi\rangle=\Re\langle z^{\beta-1}z^{1-\beta}\phi,\phi\rangle=\langle\phi,\phi\rangle\quad(z\in B_{\mathbb{C}}(r,r),\phi\in H_{0})
\]
Next, condition \eqref{enu:ass_m_n} follows from \cite[Lemma 2.1]{Picard2013_fractional},
which says 
\[
\Re\partial_{0}^{\beta}\geq\rho^{\beta}.
\]
For a proof of condition \eqref{enu:ass_n11}, we observe that by
\cite[Lemma 2.1]{Picard2013_fractional}, we have $\Re\left(\left(z^{*}\right)^{\beta}\right)^{-1}=\Re\left(z^{\beta}\right)^{-1}\geq\rho^{\beta}.$
Hence, we arrive at the following maximal regularity result for \eqref{eq:frac}.
\begin{cor}
For all $\rho>0$, $f\in L_{\rho}^{2}(\mathbb{R},H_{0}),$ the equation
\eqref{eq:frac} admits a unique solution $u\in H_{\rho}^{\beta}(\mathbb{R},H_{0})\cap D(C^{*}C)$. 
\end{cor}

\subsection*{Acknowledgments}

M.~W.~carried out this work with financial support of the EPSRC
grant EP/L018802/2: ``Mathematical foundations of metamaterials:
homogenisation, dissipation and operator theory''. This is gratefully
acknowledged.


\begin{thebibliography}{10}
\bibitem{Arendt2007} W.~{Arendt}, R.~{Chill}, S.~{Fornaro},
and C.~{Poupaud}. \newblock {$L^{p}$-maximal regularity for
nonautonomous evolution equations.} \newblock {\em {J. Differ.
Equations}}, 237(1):1--26, 2007. 

\bibitem{Arendt2014} W.~{Arendt}, D.~{Dier}, H.~{Laasri},
and E.~M. {Ouhabaz}. \newblock {Maximal regularity for evolution
equations governed by non-autonomous forms.} \newblock {\em {Adv.
Differ. Equ.}}, 19(11-12):1043--1066, 2014. 

\bibitem{Batty2015} C.~J.~K. Batty, R.~Chill, and S.~Srivastava.
\newblock Maximal regularity in interpolation spaces for second-order
cauchy problems. \newblock In W.~Arendt, R.~Chill, and Y.~Tomilov,
editors, {\em Operator Semigroups Meet Complex Analysis, Harmonic
Analysis and Mathematical Physics}, number 250 in Operator Theory:
Advances and Applications, pages 49--66. Springer International Publishing,
2015. 

\bibitem{Bu2011} S.~{Bu}. \newblock {Well-posedness of fractional
differential equations on vector-valued function spaces.} \newblock
{\em {Integral Equations Oper. Theory}}, 71(2):259--274, 2011. 

\bibitem{Chill2005} R.~{Chill} and S.~{Srivastava}. \newblock
{$L^{p}$-maximal regularity for second order Cauchy problems.}
\newblock {\em {Math. Z.}}, 251(4):751--781, 2005. 

\bibitem{Chill2008} R.~{Chill} and S.~{Srivastava}. \newblock
{$L^{p}$ Maximal regularity for second order Cauchy problems is
independent of $p$.} \newblock {\em {Boll. Unione Mat. Ital.
(9)}}, 1(1):147--157, 2008. 

\bibitem{daPrato1975} G.~{da Prato} and P.~{Grisvard}. \newblock
{Sommes d'opérateurs linéaires et équations différentielles opérationnelles.}
\newblock {\em {J. Math. Pures Appl. (9)}}, 54:305--387, 1975. 

\bibitem{Picard} R.~Picard. \newblock {A structural observation
for linear material laws in classical mathematical physics.} \newblock
{\em Math. Methods Appl. Sci.}, 32(14):1768--1803, 2009. 

\bibitem{Picard_McGhee} R.~Picard and D.~McGhee. \newblock {\em
{Partial differential equations. A unified Hilbert space approach.}}
\newblock {de Gruyter Expositions in Mathematics 55. Berlin: de
Gruyter. xviii}, 2011. 

\bibitem{Picard2016_graddiv} R.~{Picard}, S.~{Seidler}, S.~{Trostorff},
and M.~{Waurick}. \newblock {On abstract grad-div systems.}
\newblock {\em {J. Differ. Equations}}, 260(6):4888--4917, 2016. 

\bibitem{Picard2013_fractional} R.~{Picard}, S.~{Trostorff},
and M.~{Waurick}. \newblock {On evolutionary equations with material
laws containing fractional integrals.} \newblock {\em {Math.
Methods Appl. Sci.}}, 38(15):3141--3154, 2015. 

\bibitem{Picard2014_survey} R.~Picard, S.~Trostorff, and M.~Waurick.
\newblock {Well-posedness via Monotonicity. An Overview.} \newblock
In W.~Arendt, R.~Chill, and Y.~Tomilov, editors, {\em Operator
Semigroups Meet Complex Analysis, Harmonic Analysis and Mathematical
Physics}, number 250 in Operator Theory: Advances and Applications,
pages 397--452. Springer International Publishing, 2015. 

\bibitem{Ponce2013} R.~{Ponce}. \newblock {Hölder continuous
solutions for fractional differential equations and maximal regularity.}
\newblock {\em {J. Differ. Equations}}, 255(10):3284--3304,
2013. 

\bibitem{Sforza1995} D.~{Sforza}. \newblock {Maximal regularity
results for a second order integro-differential equation.} \newblock
{\em {J. Math. Anal. Appl.}}, 191(2):203--228, 1995. 

\bibitem{Trostorff2012_integro} S.~Trostorff. \newblock {On Integro-Differential
Inclusions with Operator-valued Kernels.} \newblock {\em Math.
Methods Appl. Sci.}, 38(5):834--850, 2015. 

\bibitem{Waurick2015_nonauto} M.~{Waurick}. \newblock {On non-autonomous
integro-differential-algebraic evolutionary problems.} \newblock
{\em {Math. Methods Appl. Sci.}}, 38(4):665--676, 2015. 

\bibitem{Zacher2005} R.~{Zacher}. \newblock {Maximal regularity
of type $L_{p}$ for abstract parabolic Volterra equations.} \newblock
{\em {J. Evol. Equ.}}, 5(1):79--103, 2005. \end{thebibliography}
\end{document}